\documentclass[12pt]{amsart}
\usepackage[cp1250]{inputenc}
\usepackage[T1]{fontenc}
\usepackage{amscd}

\usepackage{amsmath}
\usepackage{amsthm}
\usepackage{amssymb}
\usepackage{amscd}
\usepackage{graphicx}
\usepackage{epsfig}
\usepackage{bbm}

\theoremstyle{plain}
\newtheorem{thm}{Theorem}[section]
\newtheorem{lem}[thm]{Lemma}
\newtheorem{prop}[thm]{Proposition}
\newtheorem{cor}[thm]{Corollary}

\theoremstyle{definition}
\newtheorem{df}[thm]{Definition}

\newtheorem{exa}[thm]{Example}

\numberwithin{equation}{section}
\def\D{\mathcal D}
\def\A{\mathcal A}

\def\ss{\alpha}

\def\G{\mathcal G}

\def\a{\alpha}
\def\b{\beta}
\def\o{\Omega}
\def\g{\gamma}

\def\L{\mathcal{L}}
\def\La{\Lambda}
\def\F{\mathcal F}

\def\cc{\CC^{\infty}}

\def\R{\mathbb{R}}

\def\s{\subseteq}
\def\ddt{{\partial\over \partial t}}
\def\dds{{\partial\over \partial s}}
\def\ddtj{{\partial\over \partial t_j}}

\def\mt{\mapsto}

\def\t{\times}
\def\r{\rightarrow}
\def\rr{\rightrightarrows}
\def\cc{\CC^{\infty}}
\def\ld{,\ldots,}

\def\la{\lambda}
\def\gg{\Gamma}

\def\la{\lambda}

 \DeclareMathOperator{\sect}{Sect}
\DeclareMathOperator{\graph}{graph}
 \DeclareMathOperator{\symp}{Symp}
\DeclareMathOperator{\CC}{C}

 \DeclareMathOperator{\diff}{Diff}
 \DeclareMathOperator{\bis}{Bis}
 
 \DeclareMathOperator{\id}{id}

 \DeclareMathOperator{\supp}{supp}
\DeclareMathOperator{\dd}{d} \DeclareMathOperator{\dom}{dom}
\DeclareMathOperator{\Evol}{Evol} \DeclareMathOperator{\evol}{evol} \DeclareMathOperator{\Jac}{Jac}

\address{Faculty of Applied Mathematics, AGH University of Science and
\linebreak Technology, al. Mickiewicza 30, 30-059 Krak\'ow,
Poland} \email{tomasz.rybicki@agh.edu.pl}
\thanks{Partially supported by AGH grant n.11.420.40.}

\title{ $n$-transitivity of bisection groups of a Lie groupoid}

\author{ Tomasz Rybicki}

\begin{document}

\maketitle

\begin{abstract} The notion of $n$-transitivity can be carried over from groups of diffeomorphisms on a manifold $M$ to groups of bisections of a Lie groupoid over $M$.
The main theorem states that the $n$-transitivity is fulfilled  for all $n\in\mathbb N$ by an arbitrary group of $C^r$-bisections of a Lie groupoid $\gg$ of class $C^r$, where $1\leq r\leq\omega$, under mild conditions. For instance, the group of all bisections of any Lie groupoid and the group of all Lagrangian bisections of any symplectic groupoid are  $n$-transitive in the sense of this theorem. In particular, if $\gg$ is source connected  for any arrow $\g\in\gg$ there is a bisection passing through $\g$.
\end{abstract}

\section{Introduction}

We say that a diffeomorphism group $\G\s\diff^r(M)$ on a $C^r$-manifold $M$, where $1\leq r\leq\omega$, is \emph{$n$-transitive} if for any couple of pairwise distinct $n$-tuples $(x_1\ld x_n), (y_1\ld y_n)\in M^n$ there exists $f\in\G$ such that $f(x_i)=y_i$ for $i=1\ld n$. It was probably first observed by Milnor \cite{Mil} that the group $\diff^\infty(M)$
is $n$-transitive for all $n\in\mathbb N$, provided $\dim(M)\geq 2$. This result was generalized by Boothby \cite{Bo} for classical groups of diffeomorphisms. Next
Michor and Vizman \cite{MV} proved analogues of the previous results for the real-analytic case $r=\omega$ (see also \cite{KM}). Finally, by using the notion of generalized foliation, the problem of $n$-transitivity was solved  in a more general context in \cite{Ryb1}, which encompasses also non-transitive diffeomorphism groups. Namely, the main result
in \cite{Ryb1} states that under a natural assumption (locality) any diffeomorphism group is \emph{pseudo-$n$-transitive}, which means that the $n$-transitivity holds on each group orbit simultaneously. 

The concept of a Lie groupoid generalizes Lie groups and applies  to several situations where, comparing with the group theory, there is a lack of some symmetries (\cite{Mk}).
However, important features of Lie theory can be carried over to  Lie groupoids, e.g. basic integrability theorems. This enables to preserve important methods and results
of Lie theory in the groupoid case. On the other hand, the notion of Lie groupoid comprises or describes many basic notions of differential geometry such as manifolds, Lie groups,
vector bundles, foliations, homotopy, holonomy, monodromy, Lie group actions, gauge transformations and others. 

Given a Lie groupoid $\gg=(\gg\rr M, \a, \b)$, where $M$ is the space of units, $\a$ is the source map and $\b$ is the target map, a \emph{bisection} of $\gg$ is a smooth map
$\sigma:M\to\gg$ such that $\a\circ\sigma=\id_M$ and $\b\circ\sigma$ is a diffeomorphism of $M$.
Bisections of a Lie groupoid can be regarded as its  automorphisms. According to the Ehresmann's point of view on                                    groups in geometry the bisection groups play in it a fundamental role (see \cite{E}). 

We may generalize the notion of $n$-transitivity for Lie groupoids in the following way. A bisection group $\G$ of $\gg$ is said to be \emph{$n$-transitive} if for any two pointwise distinct $n$-tuples $(x_1\ld x_n)\in M^n$ and $(\g_1\ld\g_n)\in\gg^n$ with $\a(\g_i)=x_i$ for all $i$ and with  $\b(\g_i)\neq\b(\g_j)$ for all $i\neq j$,  there exists a bisection $\sigma\in\G$  such that $\sigma(x_i)=\g_i$ for all $i$. 

In the case of the pair groupoid over a manifold $M$ the bisections of $\gg=M\t M$ are identified  with the diffeomorphisms 
of $M$, and in this case the above definition coincides with that for diffeomorphism groups. An interesting question arises whether it is possible that a bisection group of a Lie groupoid is $n$-transitive. 
 This question was answered in the affirmative by Chen, Liu and Zhong in \cite{CLZ}
under some conditions. In the present paper we  would like to generalize  theorems from \cite {CLZ} in several aspects. Firstly we do not assume  that $\gg$ is transitive
and we consider also the case when the fibers of $\gg$
have dimension one. Next, we neither assume that the bisection group in question consists of all bisections, nor that it acts transitively on $\a$-fibers (Theorem 3.3). In particular, we obtain $n$-transitivity results for the group of  Lagrangian bisections of a symplectic groupoid.  Finally, since our method of the proof is completely different from that from \cite{CLZ}, our theorems also hold in the real-analytic case.

Let $\bis^r(\gg)$ be the group of all bisections of a Lie groupoid $\gg$ of class $C^r$, where $1\leq r\leq \omega$.  
Let $\G\leq\bis^r(\gg)$ be an arbitrary subgroup  of $\bis^r(\gg)$. The symbol $\G_0$ will stand for the totality of $\sigma\in\G$ such that there exists a $C^r$-isotopy $\{\sigma_t\}_{t\in\R}$ in $\G$ with $\sigma_0=\iota_M$, where $\iota_M: M\hookrightarrow \gg$ is the inclusion, and $\sigma_1=\sigma$. Observe that $\G_0$ is the identity component of $\G$ if $\G$ is locally contractible.
This is the case of $\bis^\infty(\gg)$ for every Lie groupoid $\gg$, since it was proved in \cite{Ryb2} and \cite{SW} that $\bis^\infty(\gg)$ is a Lie group modeled on the
space of sections of the associated Lie algebroid $A(\gg)$ of $\gg$. Likewise, in view of \cite{Ryb3} we know that for every symplectic groupoid $(\gg,\Omega)$ the group $\bis(\gg,\Omega)$ of all Lagrangian bisections of $(\gg,\Omega)$ is also a Lie group.

By $I\G$ we will denote the totality of $C^r$-mappings $$\sigma: \R\t M\ni(t,x)\mapsto \sigma_t(x)=\sigma(t,x)\in\gg$$ such that for all $t$ the mapping $\sigma_t:  x\mapsto \sigma(t,x)$ is an
element of $\G$, and $\sigma_0=\iota_M$. In other words, $I\G$ is the group of all isotopies in $\G$ starting at $\iota_M$. 

Let us formulate the (L)-condition (locality) for $\G$.

\begin{df}
A bisection subgroup $\G\leq\bis^r(\gg)$ satisfies (L)-condition, where $1\leq r\leq \infty$ (resp. $r=\omega$), if for all open, relatively compact sets $U, V\s M$
with $\overline U\s V$, and a $C^r$-isotopy $\sigma\in I\G$ sufficiently $C^1$ close to $\iota_M$, there is  a $C^r$-isotopy $\tau\in I\G$ in $\G$ with
$\tau_t=\sigma_t$ on $U$ and $\supp(\tau_t)\s V$  (resp. $\tau_t$ is sufficiently $C^1$ close to $\sigma_t$ on $U$ and
$\tau_t$ is sufficiently $C^1$ close to the inclusion map $\iota_M$ outside $V$).
\end{df}

Another concept that will be of use is the fiberwise transitivity.
\begin{df}
A group $\G\leq\bis(\gg)$ is said to be \emph{fiberwise transitive} if for every $x\in M$ there are
$k$ isotopies  $\sigma^1\ld\sigma^k\in I\G$, where $k=\dim\ss^{-1}(x)$, such that
\begin{equation}
{\partial\sigma^1\over\partial t}(x)\Big|_{t=0}=\dot\sigma^1(x)|_{t=0}\ld{\partial\sigma^k\over\partial t}(x)\Big|_{t=0}=\dot\sigma^k(x)|_{t=0}
\end{equation}
are linearly independent.
\end{df}
Let $V\s M$ be an open subset. By $\G_V$ we denote all $\sigma\in \G$ with $\supp(\sigma)\s V$, and by $I\G_V$ the totality of $\sigma\in I\G$ with $\supp(\sigma_t)\s V$ for all $t$. Next, we put $\gg_x=\a^{-1}(x)$ and denote by $\gg_x^o$ the connected component of $x$ in $\gg_x$. Let $\gg^o=\bigcup_{x\in M}\gg^o_x$. Note that $\gg_x^o$  is open in $\gg_x$
and $\gg^o$ is an open subgroupoid of $\gg$ (cf. \cite{Mk}). For any $U\s M$ let $\gg^o_U=\gg^o\cap \a^{-1}(U)$. We will also use the notation: for $n\in\mathbb N$ and for  a pairwise distinct $n$-tuple $(x_1\ld x_n)\in M^n$ we set
\begin{equation}
\mathcal R(x_1\ld x_n)=\{(\g_1\ld\g_n)\in\gg_{x_1}^o\t\cdots\t\gg_{x_n}^o|\,(\forall i\neq j), \b(\g_i)\neq \b(\g_j) \}.
\end{equation}
Finally, let $\F_\gg$ be the generalized foliation on $M$ determined by $\gg$ (see \cite{Mk}). Given $(x_1\ld x_n)\in M^n$ we denote by $L_i=\b(\gg^o_{x_i})$ the leaf of $\F_\gg$ passing through $x_i$.

Our first result concerning $n$-transitivity is the following
\begin{thm} Let $1\leq r\leq\infty$ (resp. $r=\omega$).
Suppose that  a bisection group $\G\leq\bis^r(\gg)$ satisfies (L)-condition and is fiberwise transitive.
Let $(x_1\ld x_n)\in M^n$ be a pairwise distinct $n$-tuple.  Assume, in addition, that for $i=1\ld n$ either $\dim(L_i)\geq 2$, or $x_j\not\in L_i$ for all $j\neq i$. Then for any $(\g_1\ld \g_n)\in\mathcal R(x_1\ld x_n)$  and for any $V_1\ld V_n$ such that $V_i$ is an open neighborhood of $L_i$ in $M$, there exists $\sigma\in\G$ such that $\sigma(x_i)=\g_i$ for each $i$
and $\supp(\b\circ\sigma)\s V_1\cup\ldots\cup V_n$  (resp. $\b\circ\sigma$ is sufficiently $C^1$ close to the inclusion $\iota_M$ on $M\setminus( V_1\cup\ldots\cup V_n)$).
\end{thm}

The proof will be presented in the next section.
A stronger version  of the above theorem with no assumption that $\G$ is fiberwise transitive will be proved in section 3. 

As an immediate consequence of Theorem 1.3 we have

\begin{cor}
Under the above assumptions on  $\G$ and $(x_1\ld x_n)$, suppose $(y_1\ld y_n)\in M^n$ is a pairwise distinct $n$-tuple such that $x_i, y_i$ belong to the same leaf $L_i$.   Then for any  $V_1\ld V_n$ such that $V_i$ is an open neighborhood of the leaf $L_i$, there exists $\sigma\in\G$ such that $\b(\sigma(x_i))=y_i$
and $\supp(\b\circ\sigma)\s V_1\cup\ldots\cup V_n$ for $1\leq r\leq\infty$ (resp. $\b\circ\sigma$ is sufficiently $C^1$ close to $\iota_M$ on $M\setminus( V_1\cup\ldots\cup V_n)$ for $r=\omega$).
\end{cor}

Also this corollary will be given in a more general form in section 3. Another easy consequence is the existence of bisection through a given arrow.

\begin{cor}
Under the above assumptions on  $\G$, for any $\g\in\gg^o$ there is $\sigma\in\G$ such that $\sigma(\a(\g))=\g$. Moreover, $\b\circ\sigma$ can be supported in a sufficienly
small neighborhood $V$ of the leaf $L_{\a(\g)}$ for $1\leq r\leq\infty$ (resp. $\b\circ\sigma$ can be sufficiently $C^1$ close to $\iota_M$ on $M\setminus V$ for $r=\omega$).
\end{cor}

We will show in section 3 that the fiberwise transitivity of $\G$ can be omitted.

\section{Proof of Theorem 1.3}

We start with some auxiliary result.
\begin{lem}
Let $N$ be a connected manifold and $E\s N$ be  a finite union $E=E_1\cup\ldots\cup E_l$ of submanifolds  $E_i$ of codimension $d\geq 2$ such that for each 
$\{i_1\ld i_k\}\s\{1\ld l\}$, $k>1$, the intersection $\bigcap_{s=1}^kE_{i_s}$ is a submanifold of any $E_{i_s}$ of codimension $>d$. Then $N\setminus E$  is also connected.
\end{lem}
\begin{proof} Take $x,y\in N\setminus E$. Since $N$ is a connected manifold, find a simple curve $c:[0,1]\to N$ such that $c(0)=x$ and $c(1)=y$. Then $c([0,1])\cap E$
consists of a finite number of points and closed intervals. Perform a small perturbation near each of the intersection point or interval (in the latter case by using
a finite cover of the intersection by chart domains) to obtain another curve
$\bar c: [0,1]\to N$ homotopic rel. endpoints to $c$, which has the image disjoint with $E$. It is possible thanks to the assumption that $d\geq 2$.
Thus any two points in $N\setminus E$ can be joined by a curve with image in $N\setminus E$.
This shows that $N\setminus E$ is connected.
\end{proof}

Let $(x_1\ld x_n)$ be an $n$-tuple of pairwise distinct points and $V_1\t\cdots\t V_n\s M^n$ be an open neighborhood of $L_1\t\cdots\t L_n$. 
 Set
$$ \mathcal S(x_1\ld x_n)=\{(\g_1\ld\g_n)\in\gg_{x_1}^o\t\cdots\t\gg_{x_n}^o|\,(\exists \sigma\in\G)\; \sigma(x_i)=\g_i,\, i=1\ld n\},$$
and observe that in view of the definition of bisection
$$ \mathcal S(x_1\ld x_n)\s\mathcal R(x_1\ld x_n),$$
where the set $\mathcal R(x_1\ld x_n)$ is given by (1.2). We wish to show that $\mathcal R(x_1\ld x_n)$ is a connected subset of $\gg_{x_1}^o\t\cdots\t\gg_{x_n}^o$.

Note that, after permuting indices, the set $\mathcal R(x_1\ld x_n)$ can be written in the form
$$\mathcal R(x_1\ld x_n)=\mathcal R(x_1\ld x_{k_1})\t\mathcal R(x_{k_1+1}\ld x_{k_2})\t\cdots\t\mathcal R(x_{k_{l-1}+1}\ld x_{k_l})
$$
with $l,k_s\in\mathbb N$ ($1\leq s\leq l$) and $k_l=n$, such that $L_i=L_j$ iff $k_{s-1}<i\leq j\leq k_s$ for some $s$ (here $k_0=0$), for all $L_i, L_j\in\F_\gg$.
Clearly, $\mathcal R(x)=\gg_x^o$ is connected by assumption. Therefore it suffices to show that $\mathcal R(x_1\ld x_k)$ is connected whenever $L_1=\ldots = L_k=: L$ and $\dim(L)\geq 2$.

In fact, the target projection $\b$ induces a surjective submersion $$\tilde\b: \gg_{x_1}^o\t\cdots\t\gg_{x_k}^o\to L\t\cdots\t L$$ and $$\mathcal R(x_1\ld x_k)=\gg_{x_1}^o\t\cdots\t\gg_{x_k}^o\setminus\tilde\b^{-1}(\Delta),$$ where $\Delta=\{(y_1\ld y_k)\in L^k|\,(\exists i\neq j), y_i=y_j\}$.
Since the set $\Delta$ has  codimension equal to $\dim(L)$  in $L^n$,
then so has $\tilde\b^{-1}(\Delta)$ in $\gg_{x_1}^o\t\cdots\t\gg_{x_k}^o$.  Thus, in view of Lemma 2.1 with $N=\gg^o_{x_1}\t\cdots\t\gg^o_{x_k}$ and $E=\tilde\b^{-1}(\Delta)$, the set
$\mathcal R(x_1\ld x_k)$ is connected. It follows that $\mathcal R(x_1\ld x_n)$ is also connected as a product of connected sets.

Denote $k_i=\dim(\gg_{x_i})$ for $i=1\ld n$,
and   assume the existence of bisection isotopies  $\sigma^{i1}\ld\sigma^{ik_i}\in I\G$ such that
\begin{equation}
\dot\sigma^{i1}(x_i)|_{t=0}\ld\dot\sigma^{ik_i}(x_i)|_{t=0}
\end{equation}
are linearly independent, according to (1.1). If $\dim(\gg_{x_i})=0$, i.e. $L_i=\{x_i\}$,  we assume that $x_i\not\in V_j$ for $j\neq i$. Moreover, in view 
of the (L)-condition for $1\leq r\leq\infty$ (resp. for $r=\omega$) and the fact that $\sigma^{i,j}$ may be sufficiently $C^1$-close to $\iota_M$, we can and do assume that $\supp(\sigma^{ij}_t)\s U_i\s V_i$ 
for all $i,j,t$ (resp. $\sigma^{ij}_t$  
is sufficiently $C^1$-close to $\iota_M$ outside $U_i$ for all $i,j,t$), and that $U_i$ are so small that
$$ \b(\gg^o_{U_i})\s V_i$$ for all $i$, and $$ \b(\gg^o_{U_i})\cap U_j=\emptyset$$ for all $j\neq i$.

The group multiplication in $\G$ is given by
\begin{equation}
\sigma\star\tau(x)=\sigma((\b\circ\tau(x))).\tau(x),\quad \sigma, \tau\in\G, x\in M,
\end{equation}
and the inverse of $\sigma\in\G$ by
\begin{equation}
\sigma^{-1}(x)=\sigma((\b\circ\sigma)^{-1}(x))^{-1}.
\end{equation}

For $\sigma, \tau\in I\G$ and $x\in M$ consider the mapping $$(t,s)\mapsto (\sigma_t\star\tau_s)(x).$$ In view of (2.2) we have
\begin{align}
\begin{split}
\ddt\Big|_{(0,0)}(\sigma_t\star\tau_s)(x)&=\ddt\Big|_{t=0}\sigma_t(x).x=\ddt\Big|_{t=0}\sigma_t(x)=\dot\sigma(x)|_{t=0},\\
\dds\Big|_{(0,0)}(\sigma\star\tau)(x)&=\dds\Big|_{s=0}(\b\circ\tau_s)(x).\tau_s(x)=\dds\Big|_{s=0}\tau_s(x)=\dot\tau(x)|_{s=0}.\\
\end{split}
\end{align}

It follows from (2.4) that for the mapping $(t_1\ld t_k)\mapsto(\sigma_{t_1}^1\star\cdots\star\sigma_{t_k}^k)(x)$, where $\sigma^1\ld\sigma^k\in I\G$,  we get by trivial induction that
\begin{equation}
\ddtj\Big|_{(0\ld 0)}(\sigma^1_{t_1}\star\cdots\star\sigma^k_{t_k})(x)=\dot\sigma^j(x)|_{t_j=0}.
\end{equation}
Put $N=k_1+\cdots+k_n$ and define a mapping $\Phi$ on a neighborhood of $0=(0\ld 0)\in\R^N$ into $\gg_{x_1}^o\t\cdots\t\gg_{x_n}^o$ by
\begin{displaymath}
\Phi(t_1\ld t_N)=\left(\begin{array}{ccc}
&\sigma^{11}_{t_1}\star\cdots\star\sigma^{1k_1}_{t_{k_1}}\star\cdots\star\sigma^{n1}_{t_{N-k_n+1}}\star\cdots\star\sigma^{nk_n}_{t_N}(x_1)\\
&\cdots\\
&\sigma^{11}_{t_1}\star\cdots\star\sigma^{1k_1}_{t_{k_1}}\star\cdots\star\sigma^{n1}_{t_{N-k_n+1}}\star\cdots\star\sigma^{nk_n}_{t_N}(x_n)\\
\end{array}\right).
\end{displaymath}
Set
\begin{equation}
\hat\sigma(t_1\ld t_N)=\sigma^{11}_{t_1}\star\cdots\star\sigma^{1k_1}_{t_{k_1}}\star\cdots\star\sigma^{n1}_{t_{N-k_n+1}}\star\cdots\star\sigma^{nk_n}_{t_N},
\end{equation}
and observe that, due to (2.2), $\supp(\hat\sigma(t_1\ld t_N))\s U_1\cup\cdots\cup U_n$. It follows that
\begin{equation}
\supp(\b\circ\hat\sigma(t_1\ld t_N))\s V_1\cup\cdots\cup V_n\quad \hbox{for any}\quad (t_1\ld t_N),
\end{equation}
 provided $1\leq r\leq\infty$.  Next, for $r=\omega$, by a similar argument the map $\b\circ\hat\sigma(t_1\ld t_N)$
 is sufficiently $C^1$-close to the inclusion $\iota_M$ outside
$V_1\cup\cdots\cup V_n$ for all $(t_1\ld t_N)\in\R^N$.
The mapping $\Phi$ is of class $C^r$ since $\a$ is a submersion and, consequently, the fibers are imbedded submanifolds. In the case $1\leq r\leq\infty$ the Jacobian matrix
 $\Jac\Phi(t_1\ld t_N)|_{t=0}$ of $\Phi$ at 0 is a block matrix consisting of the sequences (2.1) due to (2.5). In the case $r=\omega$ the Jacobian matrix $\Jac\Phi(t_1\ld t_N)|_{t=0}$ is $C^0$-close to the
matrix from the previous case. Consequently, in each case the Jacobian of $\Phi$ is nonzero at $0\in\R^N$, and $\Phi$ maps a neighborhood of 0 diffeomorphically onto an open neighborhood of $(x_1\ld x_n)$
in $\mathcal S(x_1\ld x_n)$.

The whole procedure may be repeated when  $(x_1\ld x_n)$ is replaced by an arbitrary 
 $n$-tuple $(\g_1\ld \g_n)\in\mathcal S(x_1\ld x_n)$, and if we take $U_1'\ld U_n'$  instead of $U_1\ld U_n$, where $U_i'\s M$ is an open neighborhood of $\b(\g_i)$
such that $U_i'$, $U'_j$ are disjoint for $i\neq j$. Furthermore, we assume for all $i$
\begin{equation} \b(\gg^o_{U'_i})\s V_i. \end{equation}
  For $i=1\ld n$ we can choose isotopies $\sigma^{ij}\in I\G_{U'_i}$ for  $j=1\ld k_i$ so that the vectors
\begin{equation}
\dot\sigma^{i1}(\b(\g_i))|_{t=0}\ld\dot\sigma^{ik_i}(\b(\g_i))|_{t=0}
\end{equation}
are linearly independent for all $i$. Then, similarly as before, we define
\begin{displaymath}
\Phi(t_1\ld t_N)=\left(\begin{array}{ccc}
&\big(\sigma^{11}_{t_1}\star\cdots\star\sigma^{1k_1}_{t_{k_1}}\star\cdots\star\sigma^{n1}_{t_{N-k_n+1}}\star\cdots\star\sigma^{nk_n}_{t_N}(\b(\g_1))\big).\g_1\\
&\cdots\\
&\big(\sigma^{11}_{t_1}\star\cdots\star\sigma^{1k_1}_{t_{k_1}}\star\cdots\star\sigma^{n1}_{t_{N-k_n+1}}\star\cdots\star\sigma^{nk_n}_{t_N}(\b(\g_n))\big).\g_n\\
\end{array}\right).
\end{displaymath}

Again, since the blocks (2.9) are linearly independent (here we use the hypothesis that $(\g_1\ld\g_n)\in \mathcal R(x_1\ld x_n)$ ), $\Phi$ is a diffeomorphism of an open neighborhood of $0\in\R^N$ onto an open neighborhood $\mathcal U$ of $(\g_1\ld\g_n)$ in $\mathcal R(x_1\ld x_n)$.
Thus we have proved that for every $(\bar\g_1\ld\bar\g_n)\in\mathcal U$ there is $\sigma\in\bis^r(\gg)$ such that
\begin{equation} \sigma(\b(\g_i)).\g_i=\bar\g_i
\end{equation}
for all $i$. Furthermore, the relation $(\g_1\ld\g_n)\sim(\bar\g_1\ld\bar\g_n)$ in the set $\mathcal R(x_1\ld x_n)$ defined by (2.10) is transitive. Indeed, for $\sigma, \tau\in\G$ satisfying (2.10) and the equality
\begin{equation*} \tau(\b(\bar\g_i)).\bar\g_i=\hat\g_i
\end{equation*}
we have
\begin{align*} \big(\tau\star\sigma(\b(\g_i))\big).\g_i&=\big(\tau(\b(\sigma(\b(\g_i)))).\sigma(\b(\g_i))\big).\g_i\\
&=\big(\tau(\b(\sigma(\b(\g_i)).\g_i)).\sigma(\b(\g_i))\big).\g_i\\
&=\tau(\b(\sigma(\b(\g_i)).\g_i)).\big(\sigma(\b(\g_i)).\g_i\big)\\
&=\tau(\b(\bar\g_i)).\bar\g_i=\hat\g_i.
\end{align*}

Thus $\sim$ is transitive and, if $(\g_1\ld\g_n)\sim(\bar\g_1\ld\bar\g_n)$ and $(\g_1\ld\g_n)\in\mathcal S(x_1\ld x_n)$ then $(\bar\g_1\ld\bar\g_n)\in\mathcal S(x_1\ld x_n)$.

The relation $\sim$ is symmetric too. In fact, for $\sigma\in\G$ verifying (2.10) we obtain from (2.3) the symmetric relation
\begin{align*} \sigma^{-1}(\b(\bar\g_i)).\bar\g_i&=\sigma^{-1}\big(\b(\sigma(\b(\g_i)).\g_i)\big).\big(\sigma(\b(\g_i)).\g_i\big)\\
&=\sigma^{-1}\big(\b(\sigma(\b(\g_i)))\big).\big(\sigma(\b(\g_i)).\g_i\big)\\
&=\sigma\big((\b\circ\sigma)^{-1}(\b(\sigma(\b(\g_i))))\big)^{-1}.\big(\sigma(\b(\g_i)).\g_i\big)\\
&=\sigma\big((\b\circ\sigma)^{-1}\circ(\b\circ\sigma)(\b(\g_i))\big)^{-1}.\big(\sigma(\b(\g_i)).\g_i\big)\\
&=(\sigma(\b(\g_i)))^{-1}.\big(\sigma(\b(\g_i)).\g_i\big)\\
&=\big(\sigma(\b(\g_i))^{-1}.\sigma(\b(\g_i))\big).\g_i=\g_i.\\
\end{align*}

In view of the transitivity and symmetry of the relation $\sim$,  $\mathcal S(x_1\ld x_n)$ and $\mathcal R(x_1\ld x_n)\setminus \mathcal S(x_1\ld x_n)$ are both open subsets of the connected set $\mathcal R(x_1\ld x_n)$. Thus $\mathcal S(x_1\ld x_n)=\mathcal R(x_1\ld x_n)$. That is, for each $n$-tuple $(\g_1\ld\g_n)\in\mathcal R(x_1\ld x_n)$ there is a bisection $\tilde\sigma$ with
$\tilde\sigma(x_i)=\g_i$.

 Finally, we have to prove that $\supp(\b\circ\tilde\sigma)\s V_1\t\cdots\t V_n$ in the case $1\leq r\leq\infty$. We proceed by induction on the number of factors 
of $\tilde\sigma$ of the form (2.6). If there is only one such a factor, that is $\tilde\sigma=\hat\sigma(t_1\ld t_N)$, then we are done by (2.7).
Suppose now that $\tilde\sigma=\sigma\star\tau$, where $\sigma=\hat\sigma(t_1\ld t_N)$ is as in (2.6), and $\supp(\b\circ\tau)\s V_1\cup\cdots\cup V_n$ by the induction assumption.
Here $\tau(x_i)=\g_i$ and $U_i'$ is a neighborhood of $\g_i$ such that (2.8) is fulfilled for all $i$. Next, the bisections $\sigma^{ij}$ in the decomposition (2.6)
 satisfy $\sigma^{ij}\in I\G_{U'_i}$ for all $i,j$. Therefore, in view of (2.2), $\supp(\b\circ\tilde\sigma)\s V_1\cup\cdots\cup V_n$,
as required.  
A similar inductive reasoning shows that  $\tilde\sigma$ may be chosen sufficiently $C^1$ close to the inclusion $\iota_M$ on
 $M\setminus( V_1\cup\ldots\cup V_n)$ for $r=\omega$. The theorem follows.

\section{The case when $\G$ is not fiberwise transitive}

In this section  we will generalize Theorem 1.3 by using the concept of singular foliations (which from now on will be called foliations) and  some related notions.
In particular, we assign to any bisection group $\G\leq\bis^r(\gg)$ a foliation on $\gg$ denoted by $\mathcal F_{\G}$.

Let $1\leq r\leq\omega$ and let $L$ be a~subset of
a~$C^r$-manifold $P$ endowed with a~$C^r$-differen\-tiable
structure which makes it an~immersed submanifold. Then $L$ is
\emph{weakly imbedded} if for any locally connected topological
space $N$ and a continuous map $f:N\rightarrow P$ satisfying
$f(N)\subset L$, the map $f:N\rightarrow L$ is continuous as well.
It~follows that in this case such a~differentiable structure is
unique. A \emph{foliation of class} $C^r$ (cf. \cite{St}) is a partition
$\mathcal{F}$ of $P$ into weakly imbedded submanifolds, called
leaves, such that the following condition holds. If~$x$~belongs to
a $k$-dimensional leaf, then there is a local chart $(U,\phi)$
with $\phi(x)=0$, and $\phi(U)=V\times W$, where $V$ is open in
$\mathbb{R}^k$, and $W$ is open in $\mathbb{R}^{n-k}$, such that
if~$L\in \mathcal{F}$ then $\phi(L\cap U)=V\times l$, where
$l=\{w\in W| \phi^{-1}(0,w)\in L\}$. A foliation is called
\emph{regular} if all leaves have the same dimension.

\begin{df} \cite{St}
A smooth mapping $\phi$ of a open subset of $\mathbb{R}\times P$
into $P$ is said to be a \emph{$C^r$-arrow}, $1\leq r\leq\omega$, if\\
(1) $\phi(t,\cdot)=\phi_{t}$ is a local $C^r$-diffeomorphism for
each $t$, possibly with empty domain,\\
(2) $\phi_0=\id$ on its domain,\\
(3) $\dom(\phi_t)\subset \dom(\phi_s)$ whenever $0\leq s<t$.
\end{df}

Given an arbitrary set of arrows $\mathcal A$, let $\mathcal A^*$
be the totality of local diffeomorphisms $\psi$ such that $\psi =
\phi(t,\cdot)$ for some $\phi\in \mathcal A$, $t\in \mathbb{R}$.
Next $\hat{\mathcal A}$ denotes the set consisting of all local
diffeomorphisms being finite compositions of elements from
$\mathcal A^*$ or $(\mathcal A^*)^{-1}=\{\psi^{-1}|\psi \in
\mathcal A^*\}$, and of the identity. Then the orbits of
$\hat{\mathcal A}$ are called \emph{accessible} sets of $\mathcal
A$.

For $x\in P$ let $\mathcal A(x)$, $\bar{\mathcal A}(x)$ be the
vector subspaces of $T_x P$ generated by
$$\{\dot{\phi}(t,y)|\phi\in \mathcal A,\phi_t(y)=x\},
\quad \{d_y\psi(v)|\psi\in \hat{\mathcal A},\psi(y)=x,v\in
\mathcal A(y)\},$$ resp. Then we have (\cite{St})
\begin{thm} Let $\mathcal A$ be an arbitrary set of $C^r$-arrows on $P$. Then
\begin{enumerate} \item every accessible set of $\mathcal A$ admits a (unique)
$C^r$-differentiable structure of~a~connected weakly imbedded
submanifold of~$P$; \item the collection of accessible sets
forms a foliation $\mathcal{F}$;  and \item
$\D(\F):=\{\bar{\mathcal A}(x)\}_{x\in P}$ is the tangent distribution of
$\mathcal{F}$.\end{enumerate}
\end{thm}

 Now let $\G\leq \bis^{r}(\gg)$ be an isotopically connected group of
bisections, i.e. any $\sigma\in\G$ can be joined to the identity by a bisection isotopy starting with $\iota_M$.  Let $\A_\G$ be the totality of restrictions of
isotopies $$\R\t \gg\ni(t,\g)\mapsto \sigma_t(\b(\g)).\g\in \gg,\quad \sigma\in I\G,$$   to open
subsets of $\R\t \gg$ in such a way that Definition 3.1 is fulfilled. Then by $\F_\G$ we denote the foliation
defined by the set of arrows $\A_\G$. Observe that $\hat\A_\G=\A_\G$
and, consequently, \begin{equation}
\bar\A_\G(\g)=\A_\G(\g).
\end{equation}
Note that each leaf of $\F_\G$ is weakly imbedded submanifold of some $\a$-fiber of $\gg$.

 Of course, any subgroup  $\G\leq \bis^{r}(\gg)$ determines
uniquely a foliation $\F_\G$. Namely, $\G$ defines uniquely a maximal isotopically connected
subgroup $\G_0$.

It is well known that a Lie groupoid $\gg\rr M$ induces a foliation on $M$, denoted by $\F_\gg$, cf. \cite{Mk}. The leaves of $\F_\gg$ are the images of the connected components of the $\a$-fibers of $\gg$ by the target projection $\b$. Now let $\G\leq\bis^r(\gg)$ and denote by $\A^M_\G$ the collection of restrictions of
isotopies $$\R\t M\ni(t, x)\mapsto \b(\sigma_t(x))\in M,\quad \sigma\in I\G,$$   to open
subsets of $\R\t M$. Then Definition 3.1 and Theorem 3.2 apply. The symbol $\F^M_\G$ stands for the foliation
defined by the set of arrows $\A^M_\G$. We also have $\hat\A^M_\G=\A^M_\G$
and $\bar\A^M_\G(x)=\A^M_\G(x)$.

Given $(x_1\ld x_n)\in M^n$ let $L_i$ be the leaf of $\F^M_\G$ passing through $x_i$. We put
$$
\mathcal R_\G(x_1\ld x_n)=\{(\g_1\ld \g_n)\in\mathcal R(x_1\ld x_n)|\,(\forall i), \b(\g_i)\in L_i\}.
$$

Now we can formulate a generalization of Theorem 1.3.

\begin{thm}
Suppose that $\G\leq\bis^r(\gg)$ satisfies (L)-condition (Definition 1.1), and let $(x_1\ld x_n)\in M^n$ be a pairwise distinct $n$-tuple.  For any $i=1\ld n$ we
assume that either $\dim(L_i)\geq 2$, or $x_j\not\in L_i$ for all $j\neq i$. Then for any $(\g_1\ld \g_n)\in \mathcal R_\G(x_1\ld x_n)$ and for any $V_1\ld V_n$ such that $V_i$ is an open neighborhood of $L_i$ in $M$, $i=1\ld n$, there exists $\sigma\in\G$ such that $\sigma(x_i)=\g_i$ for all $i$
and $\supp(\b\circ\sigma)\s V_1\cup\ldots\cup V_n$ for $1\leq r\leq\infty$ (resp. $\b\circ\sigma$ is sufficiently $C^1$ close to $\iota_M$ on $M\setminus( V_1\cup\ldots\cup V_n)$ for $r=\omega$).
\end{thm}

The proof of Theorem 3.3 is completely analogous to that of Theorem 1.3. The only difference is that instead of fiberwise transitivity we apply the fact that 
 for any $x_i\in M$ there exist
$k$ isotopies $\sigma^{i1}\ld\sigma^{ik}\in I\G$ such that
$$
\dot\sigma^{i1}(x_i)|_{t=0}\ld\dot\sigma^{ik}(x_i)|_{t=0}
$$
are linearly independent, where $k=k(i)$ is the dimension of the leaf $L_i\in\F_\G$ passing through $x_i$. In fact, this property is a consequence of Theorem 3.2(3) and the equality (3.1).
Note that the leaves $L_i$ are contained in the $M$-component $\gg^o$ of  $\gg$.

\medskip

As a consequence of Theorem 3.3 we get

\begin{cor} Suppose the hypotheses of Theorem 3.3 are fulfilled.   Then for any pairwise distinct $n$-tuple $(y_1\ld y_n)\in M^n$  such that $x_i$ and $y_i$ lie on the same leaf $L_i\in\F^M_\G$ and for any $V_1\ld V_n$ such that $V_i$ is an open neighborhood of $L_i$ in $M$, $i=1\ld n$, there is $\sigma\in\G$ such that $\b(\sigma(x_i))=y_i$ for all $i$
and $\supp(\b\circ\sigma)\s V_1\cup\ldots\cup V_n$ for $1\leq r\leq\infty$ (resp. $\b\circ\sigma$ is sufficiently $C^1$ close to $\iota_M$ on $M\setminus( V_1\cup\ldots\cup V_n)$ for $r=\omega$).
\end{cor}
Notice that  Corollary 1.4 is a special case of Corollary 3.4.

\begin{proof}  Let $L_i\in\F^M_\G$ be the leaf passing through $x_i$ and $y_i$. Since $\b|_{\gg_{x_i}^o}:\gg_{x_i}^o\to L_i$ is 
surjective we choose $(\g_1\ld\g_n)\in\mathcal R(x_1\ld x_n)$ such that  $\b(\g_i)=y_i$ for all $i$. It suffices to apply Theorem 3.3.
\end{proof}

We have the following improvement of Corollary 1.5.

\begin{cor}
If  $\G\leq\bis^r(\gg)$ satisfies (L)-condition, then for any $\g\in\gg^o$ there is $\sigma\in\G$ such that $\sigma(\a(\g))=\g$. Moreover, $\b\circ\sigma$ can be supported in a sufficienly
small neighborhood $V$ of the leaf $L_{\a(\g)}\in\F_\G^M$ for $1\leq r\leq\infty$ (resp. $\b\circ\sigma$ is sufficiently $C^1$ close to $\iota_M$ on $M\setminus V$ for $r=\omega$).
\end{cor}

\section{Regularity and the (L)-condition}

In the remaining sections we confine ourselves to the $C^\infty$-class of smoothness.  However, by applying arguments similar to those from \cite{MV} the facts and examples
presented below hold in the real-analytic case.

Let $\gg$ be a Lie groupoid and $A(\gg)$ its Lie algebroid. Let $\bis_c^\infty(\gg)$ designate the group of all compactly supported bisections of $\gg$. Next, denote by $\sect_c(A(\gg))$ the Lie algebra of all compactly supported  sections of $A(\gg)$, and by $I\sect_c(A(\gg))$ the space of all
smooth curves $\R\to\sect_c(A(\gg))$ (for the calculus on locally convex spaces, see \cite{KM} or \cite{Ne}). Note that to any bisection isotopy $\sigma\in I\bis_c^\infty(\gg)$ we can assign $X=\{X_t\}\in I\sect_c(A(\gg))$
by the following formula
\begin{equation}
X_t(x)=(R_{\sigma_t((\b\circ\sigma_t)^{-1}(x))})^*\left({\partial\sigma_s\over\partial s}((\b\circ\sigma_t)^{-1}(x))\Big|_{s=t}\right), \quad x\in M,
\end{equation}
where $R_u:\gg_x\to\gg_y$, with $\a(u)=y$ and $\b(u)=x$, is the right translation. Here $y=(\b\circ\sigma_t)^{-1}(x)$ and $u=\sigma_t(y)$. It was shown in 
\cite{Ryb2} or in \cite{SW} that the (right) evolution operator
$$\Evol^r_{\bis^\infty(\gg)}:I\sect_c(A(\gg))\ni X\mapsto\sigma\in  I\bis_c^\infty(\gg)$$
determined by (4.1) is a bijection. Moreover, the mapping
$$\evol^r_{\bis^\infty(\gg)}:I\sect_c(A(\gg))\ni X\mapsto\sigma_1\in \bis^\infty(\gg)$$
is smooth. This means that the Lie group $\bis_c^\infty(\gg)$ is \emph{regular}.

More generally, we say that a subgroup $\G\leq\bis_c^\infty(\gg)$ is \emph{regular} if there exists a closed Lie subalgebra $\L(\G)\leq\sect_c(A(\gg))$
such that defining  the  evolution operator for $\G$ as the restriction $$\Evol^r_{\G}=\Evol^r_{\bis^\infty(\gg)}\big|_{I\L(\G)},$$
where  $I\L(\G)$ stands for the space of all smooth maps $\R\to \L(\G)$, we obtain that the image of $\Evol^r_{\G}$
is in $I\G$, such that the restricted operator
$$\Evol^r_{\G}:I\L(\G)\ni X\mapsto\sigma\in  I\G$$
 is a bijection. Here we do not assume that $\G$ carries a Lie group structure.

\begin{prop}
Under the above assumption, let $\L(\G)$ possess the property: for any open subsets $U, V\s M$ with $\overline U\s V$ and $X\in I\L(\G)$ there exists $Y\in I\L(\G)$ such that
$X_t=Y_t$ on $U$ and $\supp(Y_t)\s V$ for all $t$. Then $\G$ obeys (L)-condition.
\end{prop}

The proof is straightforward. 

Therefore we have

\begin{cor} For an arbitrary Lie groupoid $\gg$ the group $\bis^\infty_c(\gg)$ (and \emph{a fortiori} $\bis^\infty(\gg)$) fulfills Theorem 1.3. In particular, for every $\g\in\gg^o$
there is $\sigma\in\bis^\infty_c(\gg)$ with $\sigma(\a(\g))=\g$ such that $\b\circ\sigma$ is supported in an arbitrarily small  neighborhood of $L_{\a(\g)}$.  
\end{cor}

\begin{proof}

In view of Theorem 3.1 it suffices to show that $\bis^\infty_c(\gg)$ is fiberwise transitive. Let $x\in M$ and $k=\dim(\gg_x)$. Take $X_1\ld X_k\in \sect_c (A(\gg))$
such that $X_1(x)\ld X_k(x)$ is a basis of $A_x(\gg)$. Then $\exp(X_1)\ld\exp(X_k)$ satisfy  Definition 1.2.
\end{proof}

\section{Symplectic groupoids and Lagrangian bisections}

Let us recall the notion of symplectic groupoid (\cite{We}, \cite{CDW}).

\begin{df} 
 A Lie groupoid $\gg$ equipped with a symplectic
form $\o$ is called \emph{symplectic} if the graph of multiplication $m$,
$\graph(m)$, is a Lagrangian submanifold of $(-\gg)\t\gg\t\gg$.
Here $-\gg$ denotes the symplectic manifold $(\gg,-\o)$.
\end{df}

Then we have

\begin{thm}(\cite{We}, \cite{CDW}, \cite{Vai}) If $(\gg,\o)$ is a symplectic groupoid then:
\begin{enumerate}
\item The inversion $i$ is an antisymplectomorphism (i.e. $i^*\o=-\o$),
and $M$ is a Lagrangian submanifold of $\gg$.

\item The foliations by fibers of $\a$ and of $\b$ are $\o$-orthogonal.

\item The smooth functions on $\gg$ constant on $\a$-fibers  and  the smooth functions on $\gg$ constant on $\b$-fibers  commutes.

\item The space of units $M$ admits a natural Poisson structure $\Lambda$ such that  $\a$ (resp. $\b$) is a Poisson morphism (resp. anti-morphism) of $(\gg,\o)$ onto
the Poisson manifold  $(M,\Lambda)$.

\item   If $\gg$ is $\a$-connected, then the   symplectic foliation $\F_{\Lambda}$ of  $(M,\Lambda)$ coincides with the foliation $\F_{\gg}$ induced by the groupoid structure.

\end{enumerate}

Such a groupoid is usually denoted by $(\gg,\o)\rightrightarrows (M,\Lambda)$.
\end{thm}

Now we need another definition of bisection, equivalent to the previous one. A bisection of a Lie groupoid $\gg$  is a submanifold $B$ of $\gg$ such that $\a|_B$ and $\b|_B$
are diffeomorphisms of $B$ onto $M$. The the set of bisections $\bis(\gg)$ is endowed the the following group product
\begin{equation}
B_1.B_2=\{\g_1.\g_2|\,(\g_1,\g_2)\in B_1\t B_2, \a(\g_1)=\b(\g_2)\}.
\end{equation}
Note that the equalities
$$ \sigma=(\a|_B)^{-1},\quad B=\sigma(M),
$$
give a bijective correspondence $\sigma\mapsto B$ between the old and new notions of bisection, and the group law (2.2) corresponds to (5.1). 
Observe that the set of all Lagrangian bisections $\bis(\gg,\Omega)$ is a subgroup
of $\bis^\infty(\gg)$.

\begin{exa}
\begin{enumerate}

\item The pair groupoid $\gg=X\t X$ with $(-\o)\oplus\o$ is a
symplectic groupoid. Then $\bis(\gg,\o)=\symp(X)$.

\item If $M$ is a manifold then $T^*M$, where $m$ is the addition in
fibers and $\pi_M=\a=\b$, is a Lie groupoid. $T^*M$ endowed
with the canonical symplectic form $\o_M=-d\la_M$ is also a symplectic
groupoid. In fact, the graph of $m$
$$
\graph(m)=\{(x_3,x_2,x_1):\, x_1+x_2-x_3=0\}.
$$
This is the image of $\mathcal N\Delta_M$, the normal bundle  of the diagonal
$\Delta_M\s M^3$ in $(T^*M)^3$, into $(T^*M)^3$ by the mapping
$(x_3,x_2,x_1)\mt(-x_3,x_2,x_1)$. Notice that $\mathcal N\Delta_M$ is
Lagrangian in $(T^*M)^3$, and the mapping is symplectic. So $T^*M$
is indeed a symplectic groupoid.

Note that $\bis(T^*M,\o)$ is the space of all closed 1-forms on $M$.
\item Let $\gg\rightrightarrows M$ be a Lie groupoid.
Then the cotangent space $T^*\gg$ equipped with 
 $\o_{\gg}=-d\lambda_{\gg}$ carries a structure of symplectic groupoid with $\mathcal N^*M$,
the conormal bundle of $M$ in $\gg$, being the space of units.
Furthermore, the canonical projection $p:T^*\gg\r\gg$ is an epimorphism
of groupoids.

\item If $G$ is a Lie group, $T^*G$ admits two symplectic groupoid
structures. First one is given as above, and the second is the
structure of transformation groupoid, where $G$ acts on $\frak g^*$
by the coadjoint action. Since these structures fulfill a compatibility
condition, $T^*G$ carries a structure of \emph{ double groupoid}, cf. \cite{CDW}, \cite{Mk}.
\end{enumerate}
\end{exa}

\begin{df} A Lie algebroid $(T^*M,\{,\}, \rho)$ over $M$ is called \emph{symplectic} if the following conditions are satisfied:

(i) the (2,0)-tensor $\Lambda$ given by $\Lambda(\theta_1,\theta_2)=\theta_2(\rho(\theta_1))$, for all 1-forms $\theta_1, \theta_2$ on $M$, is antisymmetric; and

(ii) the space of all closed 1-forms is a Lie subalgebra of $(\sect(T^*M), \{,\})$.
\end{df}

Clearly there is a one-to-one correspondence between symplectic algebroids over $M$ and Poisson structures on $M$.
Here $\rho=\Lambda^\sharp$ and
$$ \{\theta_1,\theta_2\}=\iota(\Lambda^\sharp(\theta_1))\dd\theta_2-\iota(\Lambda^\sharp(\theta_2))\dd\theta_1+\dd\Lambda(\theta_1,\theta_2).
$$
Then we have (see \cite{CDW}, \cite{Mk})
\begin{prop} 
If $(\gg,\o)\rr(M,\Lambda)$ is a symplectic groupoid then its associated algebroid $A(\gg)$ is identified with $(T^*M, \{,\}, \Lambda^\sharp)$, the symplectic algebroid of $(M,\Lambda)$. In particular, $T^*M\cong T\gg|_M/TM\cong \ker T\b$.
\end{prop}

Now we can prove the $n$-transitivity property in the symplectic case.
\begin{thm}
Let $(\gg,\Omega)\rr(M,\Lambda)$ be an arbitrary $\a$-connected symplectic groupoid. The group of Lagrangian bisections $\bis(\gg,\o)$ of  $(\gg,\o)$ is $n$-transitive for all $n\in\mathbb N$.
More precisely, let $(x_1\ld x_n)\in M^n$ be a pairwise distinct $n$-tuple and let $(\g_1\ld \g_n)\in \mathcal R(x_1\ld x_n)$. For any $V_1\ld V_n$ such that $V_i$ is an open neighborhood of $L_i$ in $M$, $i=1\ld n$, there exists $\sigma\in\bis(\gg,\o)$ such that $\sigma(x_i)=\g_i$ for each $i$
and $\supp(\b\circ\sigma)\s V_1\cup\ldots\cup V_n$.

In particular, for any $\g\in\gg$ there is $\sigma\in\bis(\gg,\o)$, with $\b\circ\sigma$ supported in an arbitrarily small neighborhood of the leaf $L_{\a(\g)}$, such that $\sigma(\a(\g))=\g$.
\end{thm}

\begin{proof}
Since the leaves of $\F_\gg$ are symplectic, the hypothesis on $(\g_1\ld\g_n)$ in Theorem 1.3 is satisfied. In fact, we have either $\dim(L_i)\geq 2$, or $L_i=\{x_i\}$.
In the latter case we assume $x_i\not\in V_j$ for $j\neq i$. Concerning the assumptions on $\G$ observe the following.
In the proof of Theorem 1.3 in place of (L)-condition and the fiberwise transitivity is only needed the following property:

\medskip

$(*)$ For any $x\in M$ and any $U$ an open neighborhood of $x$ in $M$ there are $\sigma^1\ld\sigma^k\in\sect_U(A(\gg))$, where $k=\dim(\gg_x)$, and $\sect_U$ denotes sections
supported in $U$, such that the equalities (1.1) hold. 

\medskip

Note that in our situation $\dim(\gg)=2k$ with $\dim(M)=k$ and $\dim(\gg_x)=k$. Obviously the condition $(*)$ holds for $(\gg,\o)$. Let $(x_1\ld x_n)$ be an $n$-tuple of pairwise distinct points and $V_1\t\cdots\t V_n\s M^n$ be an open neighborhood of $L_1\t\cdots\t L_n$. In view of $(*)$ we can find smooth  functions $u^{ij}\in \cc(M,\R)$, where $i=1\ld n, j=1\ld k$ such that $\supp(u^{ij})\s U_i$ with $\b(\gg^o_{U_i})\s V_i$ for all $i,j$, and the vector fields
$$ X^{ij}=\o^\sharp(\dd u^{ij})$$ are such that
$$ X^{i1}(x_i)\ld X^{ik}(x_i)$$ constitutes a basis of $A_{x_i}(\gg)$. Then we put $\sigma^{ij}=\exp(X^{ij})$ and the rest of the proof is the same as that of Theorem 1.3.
\end{proof}

Observe that for any $\sigma\in\bis(\gg,\o)$ the mapping $M\ni x\mapsto(\b\circ\sigma)(x)\in M$
is a Poisson diffeomorphism of $(M, \La)$. Analogously as for general groupoids, but without the assumption on $(x_1\ld x_n)$,  we get

\begin{cor}
Under the above assumption,
let $(x_1\ld x_n), (y_1\ld y_n)\in M^n$ be  pairwise distinct $n$-tuples such that $x_i, y_i$ belong to the same leaf of the foliation $\F_\gg=\F_\Lambda$ 
on $M$.   Then for any  $V_1\ld V_n$ such that $V_i$ is an open neighborhood of the leaf $L_i\in\F_\gg$, $i=1\ld n$, there exists $\sigma\in\bis(\gg, \o)$ such that the Poisson diffeomorphism $\b\circ\sigma$ satisfies $(\b\circ\sigma)(x_i)=y_i$
and $\supp(\b\circ\sigma)\s V_1\cup\ldots\cup V_n$.
\end{cor}

\end{document}